\documentclass[11pt]{article}

\title{New examples of period collapse} \author{Dan
  Cristofaro-Gardiner\footnote{Mathematics Department, Harvard
    University, Cambridge MA, USA.  Electronic address:
    gardiner@math.harvard.edu}, Teresa Xueshan Li\footnote{Mathematics
    and Statistics, Southwest University, Chongqing,
    P.R. China. Electronic address: pmgb@swu.edu.cn} , and Richard
  Stanley\footnote{Department of Mathematics, Massachussets Institute
    of Technology, Cambridge MA, USA.  Electronic address:
    rstan@math.mit.edu.}} \date{}

\usepackage{amssymb}
\usepackage{latexsym}
\usepackage{amsmath}
\usepackage{amsthm}
\usepackage{enumerate}
\usepackage{amscd}
\usepackage{graphicx}
\numberwithin{equation}{section}
\numberwithin{figure}{section}

\newtheorem{theorem}{Theorem}[section]
\newtheorem{proposition}[theorem]{Proposition}
\newtheorem{corollary}[theorem]{Corollary}
\newtheorem{lemma}[theorem]{Lemma}
\newtheorem{lemma-definition}[theorem]{Lemma-Definition}

\newtheorem{claim}[theorem]{Claim}

\theoremstyle{definition}

\newtheorem{remark}[theorem]{Remark}

\newtheorem{example}[theorem]{Example}

\newcommand{\eqdef}{\;{:=}\;}

\newcommand{\op}{\operatorname}

\newcommand{\st}{\,|\,}

\newcommand{\bpm}{\begin{pmatrix}}
\newcommand{\epm}{\end{pmatrix}}

\renewcommand{\epsilon}{\varepsilon}

\begin{document}

\setcounter{tocdepth}{2}

\maketitle

\begin{abstract}
``Period collapse" refers to any situation where the period of the
  Ehrhart function of a polytope is less than the denominator of that
  polytope.  We study several interesting situations where this
  occurs, primarily involving triangles.  For example: 1) we determine
  exactly when the Ehrhart function of a right triangle with legs on
  the axes and slant edge with irrational slope is a polynomial; 2) we
  find triangles with periods given by any even-index k-Fibonacci
  number, and larger denominators; 3) we construct several higher
  dimensional examples.  Several related issues are also discussed,
  including connections with symplectic geometry.
\end{abstract}

\section{Introdution}

\subsection{Background}

Let $\mathcal{P} \subset \mathbb{R}^{d}$ be a convex polytope.  The
counting function
\[ I_{\mathcal{P}}(t):=\#(t\mathcal{P}\cap \mathbb{Z}^{d})\]
for a positive integer $t$ is called the {\em Ehrhart function} of
$\mathcal{P}$.  A classical result of Ehrhart \cite{Ehrhart1977}
asserts that when $\mathcal{P}$ is rational, $I_{\mathcal{P}}(t)$ is a
``quasipolynomial" in $t$.  This means that the function
$I_{\mathcal{P}}(t)$ is a polynomial in $t$, with periodic
coefficients of integral period.  The minimum common period of these
coefficients is called the {\em period} of $\mathcal{P}$.  While it is
known that the period of $\mathcal{P}$ is bounded from above by the
minimum integer $\mathcal{D}$ such that the vertices of $\mathcal{D}
\cdot \mathcal{P}$ are integral, called the {\em denominator} of
$\mathcal{P}$, the precise relationship between $\mathcal{P}$ and its
period can be quite subtle.

For example, in their study of vertices of Gelfand-Tsetlin polytopes,
De Leora and McAllister \cite{Loera-McAllister2004} constructed an
infinite family of non-integral polytopes for which the Ehrhart
function is still a polynomial. Later, McAllister and Woods
\cite{McAllister-Woods2005} extended this result to any dimension
$d\geq 2$. They showed that, given $\mathcal{D}$ and $s$ such that
$s|\mathcal{D}$, there exists a $d$-dimensional polytope with
denominator $\mathcal{D}$ whose Ehrhart quasi-polynomial has period
$s$.  Other interesting related work appears in (for example)
\cite{Beck-Robins-Sam2009, Haase-McAllister2007, Gardiner-Kleinman,
  Woods2005}.

Any situation where the period of $\mathcal{P}$ is smaller than its
denominator is called {\em period collapse}.  In this paper, we
further study this phenomenon through several interesting examples,
which we now explain.

\subsection{Irrational triangles}

As alluded to above, traditionally the objects of study in Ehrhart
theory are {\em rational} polytopes.  The first question we are
concerned with here is how frequently an irrational polytope has an
Ehrhart function that is a quasi-polynomial or a polynomial.  We
should think of the denominator of an irrational polytope as being
infinite, so our question is about a particularly extreme form of
period collapse.

An interesting class of examples comes from fixing positive numbers
$u$ and $v$ with $u/v$ irrational, and studying the Ehrhart function
of the triangle $\mathcal{T}_{u,v} \subset \mathbb{R}^2$ with vertices
$(0,0), (1/u,0)$, and $(0,1/v)$.  It turns out that one can completely
determine when the Ehrhart function of such a polytope is a
quasipolynomial or a polynomial.  To state our result, first recall
that any polytope whose Ehrhart function is a polynomial is called
{\em pseudo-integral}.  In analogy with this, we will call an
(irrational) polytope {\em pseudo-rational} if its Ehrhart function is
a quasipolynomial.  Of course, if $\mathcal{T}$ is pseudo-rational,
then any scaling of $\mathcal{T}$ by a positive integer is as well; we
will therefore call a pseudo-rational triangle {\em primitive} if no
scaling $1/t \cdot \mathcal{T}$ for an integer $t>1$ is
pseudo-rational.

We can now state precisely which triangles in the family
$\mathcal{T}_{u,v}$ are pseudo-rational and pseudo-integral.  In fact,
$u$ and $v$ must be certain special conjugate quadratic
irrationalities:

\begin{theorem}
\label{thm:irrationalperiodcollapse}
Let $u$ and $v$ be positive numbers with $u/v$ irrational.
\begin{enumerate}[(i)]
\item  The triangle $\mathcal{T}_{u,v}$ is primitive and
  pseudo-rational if and only if
\begin{align}
\label{eqn:keyequation}
u+v=\alpha \\
1/u+1/v = \beta, \nonumber
\end{align}
for positive integers $\alpha, \beta$.  The period of this
quasipolynomial divides $\alpha$.
\item The triangle $\mathcal{T}_{u,v}$ is primitive and
  pseudo-integral if and only if \eqref{eqn:keyequation} is satisfied,
  and in addition, either $\alpha=1$ or
\[(\alpha,\beta) \in \lbrace (3,3),(2,4) \rbrace.\]
\end{enumerate}
\end{theorem}

To simplify the notation, we call a pair of positive numbers $(u,v)$
with $u/v$ irrational and $(u,v)$ satisfying \eqref{eqn:keyequation}
{\em admissible}, and we also call a triangle $\mathcal{T}_{u,v}$
admissible when $(u,v)$ is an admissible pair.  To get a feel for
Theorem~\ref{thm:irrationalperiodcollapse}, the following example,
which we prove in \S\ref{sec:mainthm}, is illustrative.

\begin{example}
\label{ex:smallesttriangle}
Let $u/v$ be irrational.  The pseudo-integral triangle in the family
$\mathcal{T}_{u,v}$ with smallest area corresponds to
$(u,v)=(\tau^2,1/\tau^2)$, where $\tau = \frac{1 + \sqrt{5}}{2}$ is
the Golden Ratio.
\end{example}

One of the key steps in the proof of the ``only if" direction of the
first bullet point involves a slightly stronger statement than what is
required, which is of potentially independent interest.  Recall that a
sequence $f(n)$ is $\mathcal{P}$-{\em recursive}, of order $k$, if
there are polynomials $p_0,\ldots,p_k$, not all 0, such that the
recurrence relation
\[ p_k(n+k)f(n+k) + \cdots + p_0(n)f(n) = 0\]
holds for all nonnegative integer $n$.  In general, it can be
difficult to show that a sequence is not $\mathcal{P}$-recursive.
However, natural examples of sequences which are not
$\mathcal{P}$-recursive are given by the following.

\begin{theorem}
\label{thm:precursion}
Let $u$ and $v$ be positive numbers with $u/v$ irrational, and assume
that $1/u + 1/v$ and $u+v$ are rational, but $\mathcal{T}_{u,v}$ is
not a positive integer scaling of a primitive pseudo-rational
triangle.  Then the sequence $f(n) \eqdef  I_{\mathcal{T}_{u,v}}(n)$
is not $\mathcal{P}$-recursive.
\end{theorem}

To introduce our final set of results, note that since the Ehrhart
functions for pseudo-rational $\mathcal{T}_{u,v}$ are
quasipolynomials, one can ask to what degree some of the basic results
from Ehrhart theory in the rational case apply.  In fact, versions of
Ehrhart-Macdonald reciprocity, as well as the nonnegativity theorem and
monotonicity theorem of the third author, hold for these triangles; see
Proposition~\ref{prop:properties}.

Although our primary interest here is for triangles, we can also give
examples of irrational polytopes with quasipolynomial Ehrhart
functions in any dimension; see Example~\ref{ex:1dex},
Example~\ref{ex:ndex}, and Example~\ref{ex:3dex}.

\subsection{Criteria for period collapse for rational triangles}

When $(u,v)$ are rational, the period collapse question for
$\mathcal{T}_{u,v}$ is less well understood than in the irrational
case.  Nevertheless, we find many new examples of rational triangles
of this form exhibiting significant period collapse.  The key is the
following criterion.

\begin{theorem}
\label{thm:main theorem-1}
Let $u=q/p$ and $v=s/r$ in lowest terms.  Then $q$ is a period of the
Ehrhart quasipolynomial for $\mathcal{T}_{u,v}$ if
 \begin{align}\label{eqn:collapse-condition-1}
 s | p,\ \ p|(rq+1),\ \   {\rm and}\ \ \
  \gcd\left(\frac{rq+1}{p},s\right)=1.
 \end{align}
\end{theorem}

For example, if $q=1$, then one obtains the McAllister and Woods
example of period collapse mentioned above as a corollary of
Theorem~\ref{thm:main theorem-1}.  Indeed, the theorem implies that
the triangle with vertices $(0,0)$, $(p,0)$ and
$\left(0,\frac{p-1}{p}\right)$ is a pseudo-integral triangle with
denominator $p$. This triangle is unimodularly equivalent to the
pseudo-integral triangle found by McAllister and Woods \cite[Theorem
  2.2]{McAllister-Woods2005}, which has vertices $(0,0)$, $(p,0)$ and
$\left(1,\frac{p-1}{p}\right)$, via the map
\begin{align*}
\varphi(x)=x\left(
              \begin{array}{cc}
                -1 & 0 \\
                -p & 1 \\
              \end{array}
            \right)+(p,0).
\end{align*}

Theorem~\ref{thm:main theorem-1} can be used to construct other
pseudo-integral triangles, via the following result.

\begin{corollary}\label{pseudo-Corollary}
Let $u=q/p, v=s/r$ in lowest terms.  The triangle $\mathcal{T}_{u,v}$
is pseudo-integral if
 \begin{align}\label{pseudo-condition-1}
 s|p,\ \ p|(rq+1),\ \ \gcd\left(\frac{rq+1}{p},s\right)=1
 \end{align}
 and\
 \begin{align}\label{pseudo-condition-2}
 q|r,\ \ r|(sp+1),\ \ \gcd\left(\frac{sp+1}{r},q\right)=1.
 \end{align}
\end{corollary}

The criteria of \eqref{eqn:collapse-condition-1} also have a nice
relationship with the $k$-Fibonacci numbers.  Specifically, we can use
Theorem~\ref{thm:main theorem-1} to construct triangles with period dividing
any even-index $k$-Fibonacci number, and high denominator
(Theorem~\ref{period-collapse-3}).  If $s=p$ and $r=q$, then the
condition \eqref{eqn:collapse-condition-1} is also sufficient for $q$
to be a period, which we also show (Theorem~\ref{thm2}).

\subsection{Relationship with symplectic geometry}

We briefly remark that the triangles $\mathcal{T}_{u,v}$ with $u/v$
irrational satisfying \eqref{eqn:keyequation} seem to have interesting
relationships with symplectic geometry.  For example, the triangle
from \eqref{ex:smallesttriangle} is closely related to a foundational
result of McDuff and Schlenk \cite{ms} about ``symplectic" embedding
problems; see also \cite{Gardiner-Kleinman}.  Some of the other
triangles from \eqref{eqn:keyequation} also seem to be relevant in the
context of symplectic embeddings.  This is further explored in
\cite{hm}.

\subsection{Acknowledgements}

We thank Bjorn Poonen for his help with Lemma~\ref{lem:poonenlemma}.
The first author would also like to thank Martin Gardiner for helpful
discussions. The first author was partially supported by NSF grant DMS-1402200. The second author was
supported by the Research Fund
for the Doctoral Program of Higher Education of China (Grant No. 20130182120030), the Fundamental
Research Funds for Central Universities (Grant No. XDJK2013C133) and the China
Scholarship Council. The third author was partially supported by NSF grant
DMS-1068625.

\section{Irrational triangles with Ehrhart quasipolynomials}

\subsection{Proof of the main theorem}
\label{sec:mainthm}

Here we prove Theorem~\ref{thm:irrationalperiodcollapse} in several steps.

\begin{proof}

{\em Step 1.}  We first prove the ``if" direction of the first bullet point.

To start, we want to show that the number of nonnegative integer
solutions $(x,y)$ to
\begin{equation}
\label{eqn:latticepointcount}
ux + vy \le t
\end{equation}
is a quasipolynomial in $t$ for nonnegative integer $t$.  Let $0 \le m
\le \lfloor t/\alpha \rfloor$ be an integer.  By
\eqref{eqn:keyequation}, the number of solutions to
\eqref{eqn:latticepointcount} with $x=m$ and $y \ge m$ is $1 + \lfloor
\frac{t - m\alpha}{v} \rfloor$.  Similarly, the number of solutions to
\eqref{eqn:latticepointcount} with $y = m$ and $x > m$ is $\lfloor
\frac{t-m\alpha}{u} \rfloor$.  It follows that the number of solutions
to \eqref{eqn:latticepointcount} is
\begin{equation}
\label{eqn:sum}
\sum^{\lfloor \frac{t}{\alpha} \rfloor}_{m = 0}\left( 1+\left\lfloor
\frac{t-m\alpha}{v}\right\rfloor + \left\lfloor \frac{t-m\alpha}{u}
\right\rfloor\right).
\end{equation}

By \eqref{eqn:keyequation}, we know that
\begin{align*}
\left\lfloor \frac{t-m\alpha}{v} \right\rfloor &= \lfloor
(t-m\alpha)(\beta-1/u) \rfloor \\
& = \left\lfloor (t\beta-m\alpha\beta)-\frac{(t-m\alpha)}{u}
\right\rfloor \\
& = (t\beta-m\alpha\beta)+\left\lfloor -\frac{(t - m\alpha)}{u}
\right\rfloor.
\end{align*}

We can therefore rewrite \eqref{eqn:sum} as
\begin{equation}
\label{eqn:simplerform}
\sigma(t) + \sum_{m = 0}^{\lfloor \frac{t}{\alpha} \rfloor}
(t\beta-m\alpha\beta),
\end{equation}
where
\begin{flalign*}
\sigma(t) \eqdef \left\{\begin{aligned}1 & \qquad\textrm{if }t\textrm{ is divisible by $\alpha$,} \\
0 & \qquad\textrm{otherwise}.
\end{aligned}
\right.
\end{flalign*}

We now claim that \eqref{eqn:simplerform} is a quasipolynomial in $t$,
of period dividing $\alpha$.   We can rewrite \eqref{eqn:simplerform}
as
\begin{equation}
\label{eqn:simplerequation}
\sigma(t) + (\lfloor t/ \alpha \rfloor + 1)\left(t\beta -
\frac{\alpha\beta \lfloor t/\alpha \rfloor}{2}\right).
\end{equation}

Let $z \eqdef t$ (mod $\alpha$).  Then $\lfloor \frac{t}{\alpha}
\rfloor=\frac{t-z}{\alpha}.$  Hence we can rewrite
\eqref{eqn:simplerequation} as
\[\sigma(t) + \frac{(t-(z-\alpha))(t\beta+z\beta)}{2\alpha}\]
\begin{equation}
\label{eqn:explicitform}
 = \frac{t^2\beta+t\alpha\beta+z\beta(\alpha-z)+2\alpha\sigma(t)}{2\alpha}.
\end{equation}

The coefficients of $t^2$ and $t$ in \eqref{eqn:explicitform} do not
depend on $t$, and the constant term only depends on the equivalence
class of $t$ modulo $\alpha$.  This proves the pseudo-rational part of
the ``if" direction of the first bullet point.

{\em Step 2.}  We now begin the proof of the ``only if" direction (we
will return to the rest of the proof of the ``if direction"" in a
later step).  Let $(u,v)$ be any pair of positive numbers with $v/u$
irrational.

We will use the following, which was explained to us by Bjorn Poonen:

\begin{lemma}
\label{lem:poonenlemma}
If $u/v$ is irrational, then $I_{\mathcal{T}_{u,v}}(t)=
\frac{1}{2uv}t^2 +\frac{1}{2}\left(\frac 1u+\frac 1v\right)t+o(t)$ for
$t \in \mathbb{R}_{> 0}.$
\end{lemma}

\begin{proof}

By scaling, we may assume that $u=1$.  By counting in each vertical
line, we see that
\begin{equation}
\label{eqn:eqn1}
I_{\mathcal{T}_{1,1/v}}(t) = \sum_{m=0}^{\lfloor t \rfloor}(\lfloor
t/v-m/v \rfloor+1).
\end{equation}
It is convenient to define a function $f(x)$ by
\[ \lfloor x \rfloor + 1 = x + f(x).\]
We can then rewrite \eqref{eqn:eqn1} as
\begin{equation}
\label{eqn:eqn2}
I_{\mathcal{T}_{1,v}}(t)=\sum_{m=0}^{\lfloor t
  \rfloor}(t/v-m/v)+\sum_{m=0}^{\lfloor t \rfloor} f(t/v-m/v).
\end{equation}

Now note that the $m$th term in the first sum computes the area of
the trapezoid defined by $m - 1/2 \le x \le m + 1/2, 0 \le y \le
t/v-x/v.$   The first sum is therefore within $O(1)$ of the area of
the triangle defined by $-1/2 \le x \le t, 0 \le y \le t/v-x/v,$ and
so we have
\begin{equation}
\label{eqn:firstsum}
\sum_{m=0}^{\lfloor t \rfloor}(t/v-m/v)=(t^2+t)/2v+O(1).
\end{equation}

The second sum is a sum of values of a bounded, integrable, periodic
function $f$ at points that become equidistributed mod $1$ as $t \to
\infty$ (by Weyl's criterion for uniform distribution), and so we find
that
\begin{equation}
\label{eqn:secondsum}
\sum_{m=0}^{\lfloor t \rfloor} f(t/v-m/v)=t \int_0^1 f(x)dx + o(t) =
t/2+o(t).
\end{equation}

Lemma~\ref{lem:poonenlemma} now follows by combining \eqref{eqn:eqn2},
\eqref{eqn:firstsum}, and \eqref{eqn:secondsum}.

\end{proof}

{\em Step 3.} Now assume that $I_{\mathcal{T}_{u,v}}(t)$ is a
quasipolynomial in the positive integer $t$, of period $C$.  Write
$I_{\mathcal{T}_{u,v}}(pt)=A(C^2t^2)+B(Ct)+D.$  We know that the
number  $I_{\mathcal{T}_{u,v}}(t)$ must always be an integer.  Hence,
the numbers $A$ and $B$ here must be rational.
Lemma~\ref{lem:poonenlemma} now implies that $u+v$ and $1/u+1/v$ must
be rational as well.

So, if $\mathcal{T}_{u,v}$ is pseudo-rational, we know that $u+v$ and
$1/u+1/v$ must be rational; write $u+v = \tilde{\alpha}, 1/u+1/v =
\tilde{\beta}.$  The ``only if" direction will then follow from the
following claim, which is closely related to
Theorem~\ref{thm:precursion}:

\begin{claim}
\label{clm:importantclaim}
Unless $\tilde{\beta}$ and $\tilde{\alpha}\tilde{\beta}$ are both
integers, the Ehrhart function of $\mathcal{T}_{u,v}$ is not
$\mathcal{P}$-recursive.
\end{claim}

\begin{proof}[Proof of claim]

By \eqref{eqn:sum} (which still holds even though $\tilde{\alpha},
\tilde{\beta}$ are no longer assumed integral), we can write
$I_{\mathcal{T}_{u,v}}(t)$ as
\[ \sum^{\lfloor \frac{t}{\tilde{\alpha}} \rfloor}_{m = 0}\left(
1+\left\lfloor \frac{t-m\tilde{\alpha}}{v}\right\rfloor + \left\lfloor
\frac{t-m\tilde{\alpha}}{u} \right\rfloor\right).\]
We know that $\left\lfloor
\frac{t-m\tilde{\alpha}}{v}\right\rfloor=\left\lfloor
\tilde{\beta}(t-m\tilde{\alpha}) - \frac{t-m\tilde{\alpha}}{u}
\right\rfloor$.  Hence we can rewrite this sum as
\begin{equation}
\label{eqn:rewrittensum}
\op{err}(t)+\sigma(t)+\sum^{\lfloor \frac{t}{\tilde{\alpha}}
  \rfloor}_{m = 0} \lfloor \tilde{\beta}(t-m\tilde{\alpha}) \rfloor,
\end{equation}
where
\begin{flalign*}
\sigma(t) \eqdef \left\{\begin{aligned}1, & \qquad\textrm{if
}t/\tilde{\alpha}\textrm{ is an integer,} \\
0, & \qquad\textrm{otherwise},
\end{aligned}
\right.
\end{flalign*}
and
\[ \op{err}(t) \eqdef \# \left\lbrace 0 \le m \le \lfloor
t/\tilde{\alpha} \rfloor | \lbrace \tilde{\beta}(t-m\tilde{\alpha})
\rbrace > \left\lbrace \frac{t-m\tilde{\alpha}}{u} \right\rbrace
\right\rbrace,\]
where $\lbrace \cdot \rbrace$ denotes the fractional part function.

Now assume that $I_{\mathcal{T}_{u,v}}(t)$ is $\mathcal{P}$-recursive, and write the
recurrence as
\begin{equation}
\label{eqn:recursiondefn}
p_s(t)\Big(\op{err}(t) + q(t)\Big) +  \cdots + p_0(t)
\Big(\op{err}(t-s)+q(t-s)\Big) = 0,
\end{equation}
where $q$ is given by $q \eqdef
\sigma(t)+\sum^{\lfloor \frac{t}{\tilde{\alpha}} \rfloor}_{m = 0}
\lfloor \tilde{\beta}(t-m\tilde{\alpha}) \rfloor.$  The function $q$ is
a quasipolynomial in $t$.    Now write
$\tilde{\alpha} = \frac{p}{q}$ in lowest terms and introduce the
function
\[ \op{head}(t) = \# \left\lbrace 0 \le m \le (q-1) |   \lbrace
\tilde{\beta}(t-m\tilde{\alpha}) \rbrace > \left\lbrace
\frac{t-m\tilde{\alpha}}{u} \right\rbrace \right\rbrace.\]
We know by its definition that $\op{err}(t) - \op{err}(t-p) =
\op{head}(t)$ for all $t \ge p$.  Then for $t \ge p+s$, by applying \eqref{eqn:recursiondefn} twice, we have
\begin{equation}
\label{eqn:recursion}
\begin{aligned}
 & \qquad p_s(t) \op{head}(t) + \cdots + p_0(t)
  \op{head}(t-s)
   + p_s(t)q(t)\\ & -p_s(t-p)q(t-p) + \cdots+p_0(t)q(t-s)
     -p_0(t-p)q(t-p-s)\\ & + \op{err}(t-p)(p_s(t)-p_{s}(t-p)) + \ldots + \op{err}(t-s-p)(p_0(t)-p_0(t-p))  = 0.
\end{aligned}
\end{equation}  We will now derive a contradiction.

{\em Step 4.}  Assume first that $\tilde{\beta}$ is not an integer, and write
$\tilde{\beta}=k/l$.  Let $C$  be the period of the quasipolynomial $q$.
Introduce the set $S = \lbrace 1+iCl\,|\,  i \in \mathbb{Z}_{\ge 0}
\rbrace$.  Then for any $t \in S$, $\lbrace t\tilde{\beta} \rbrace$ is some
fixed nonzero number $a_0$ independent of $t$.  We now claim there is
some $\epsilon > 0$ with the property that if $\lbrace \frac{t}{u}
\rbrace$ is in $(a_0-\epsilon,a_0+\epsilon)$, then $\op{head}(t)$ is
determined by whether or not $a_0 > \lbrace \frac{t}{u} \rbrace$, and
$\op{head}(t)$ will differ depending on whether or not this condition
is met.

To see this, introduce the two homeomorphisms $f_1, f_2$ from $[0,1]$
(mod $1$) to itself given by:
\[ f_1(x) = \lbrace x - \tilde{\alpha}\tilde{\beta} \rbrace, \quad
\quad f_2(x) = \left\lbrace x - \frac{\tilde{\alpha}}{u} \right\rbrace
.\]
Since $\tilde{\alpha}$ is rational, while $u$ is irrational, for
rational $x$ we can never have $f_1^m(x)=f_2^m(x)$ for any positive
integer $m$.  With this in mind, consider
$f_1(a_0),\ldots,f_1^{q-1}(a_0)$.  If we take $\epsilon$ sufficiently
small, we can guarantee that for any $y \in (a_0 - \epsilon, a_0 +
\epsilon)$, $f_2^i(y) \ne f_1^i(a_0)$ for any $1 \le i \le q-1$.  By
shrinking $\epsilon$ if necessary, we can also conclude that for any
$y$ in this interval, $f_2^i(y) \ne 0$.  It follows that if $\lbrace
\frac{t}{u} \rbrace$ is in $(a_0 - \epsilon, a_0 + \epsilon)$, then
$\op{head}(t)$ is determined by whether or not $a_0 > \lbrace
\frac{t}{u} \rbrace$; to emphasize, $\op{head}(t)$ will be different
depending on whether or not this condition is met, as claimed.

{\em Step 5.}  Assume now that $t \in S$, and $t \ge s + p$.

\vspace{3 mm}

{\em Claim 1.}  By decreasing $\epsilon$ if necessary, if $\lbrace
\frac{t}{u} \rbrace$ is in $(a_0-\epsilon,a_0+\epsilon)$, then
$\op{head}(t\pm1),\ldots,\op{head}(t\pm s)$ do not depend on $t$.

{\em Proof of Claim}.  Let $a^{\pm}_1,\ldots,a^{\pm}_s$ be the rational numbers
(mod $1$) defined by $a^{\pm}_i \eqdef \lbrace \tilde{\beta}(t\pm i) \rbrace =
\lbrace a_0 \pm i\tilde{\beta} \rbrace$.  For any $1 \le i \le s$, we
can never have $\lbrace a_0 \pm \frac{i}{u} \rbrace = a^{\pm}_i$, since
$\lbrace a_0 \pm \frac{i}{u} \rbrace$ is irrational.  Consider then
$a^{\pm}_i$ and $z^{\pm}_i= \lbrace a_0 \pm \frac{i}{u}\rbrace$.  Essentially the
same argument\footnote{To elaborate slightly on this, it is probably
  worth emphasizing that we can never have $\lbrace a^{\pm}_i -
  m\tilde{\alpha}\tilde{\beta} \rbrace =\lbrace z^{\pm}_i-\frac{m
    \tilde{\alpha}}{u} \rbrace$ for $0 \le m \le (q-1)$.} as in Step
$5$ allows us to conclude that if $\tilde{\epsilon}$ is sufficiently
small, and $x_i$ is some irrational number in $(z^{\pm}_i
-\tilde{\epsilon},z^{\pm}_i +\tilde{\epsilon})$, then $\# \left\lbrace 0 \le
m \le (q-1) |   \lbrace a^{\pm}_i -m \tilde{\alpha}\tilde{\beta} \rbrace >
\lbrace x_i - \frac{m\tilde{\alpha}}{u} \rbrace \right\rbrace$ does
not depend on $x_i$.  Since we can make $\frac{t \pm i}{u}$ arbitrarily
close to $z^{\pm}_i$ by making $\lbrace \frac{t}{u} \rbrace$ sufficiently
close to $a_0$, the claim follows.

\vspace{3 mm}
We can now complete the proof, in the case where $\tilde{\beta}$ is
not an integer.  Let $p_{s-\tilde{s}}$ be one of the polynomials $p_j$, with the property
that no other $p_j$ has higher degree.   As $t$ ranges over $S$, $\lbrace t/u\rbrace$ is dense
in $(0,1)$.  Take an infinite sequence $t_i$ such that $\lbrace
(t_i-\tilde{s})/u\rbrace \in (a_0,a_0+\epsilon)$ and $t_i - \tilde{s} \in S$.  Now it follows from Claim $1,$ \eqref{eqn:recursion}, and the fact that we are fixing $t$, mod C, that if $i$ is sufficiently large (so that $p_{s-\tilde{s}}(t_i) \ne 0$) then
\begin{equation}
\label{eqn:rationalfunctionequation}
\op{head}(t_i-\tilde{s}) = a(t_i)/p_{s-\tilde{s}}(t_i) - R(t_i)/p_{s-\tilde{s}}(t_i),
\end{equation}
where $a$ is some fixed polynomial of $t_i$, whose coefficients do not depend on $i$.   Meanwhile, the term $R(t)$ is given by
\[ R(t_i) = \sum^s_{j=0} \op{err}(t_i-p-j)(p_{s-j}(t_i)-p_{s-j}(t_i-p)).\]

The term $R(t_i)/p_{s-\tilde{s}}(t_i)$ is controlled by the following:

{\em Claim 2.} Let $x_l \in S$ be a sequence of points tending to $+\infty$.  Then
\[\lim_{l \to \infty} R(x_l)/p_{s-\tilde{s}}(x_l) \to M,\]
where $M$ is some constant.

\begin{proof}[Proof of Claim]

We know that the degree of $p_{s-\tilde{s}}$ is strictly greater than the degree of any of the polynomials $p_{s-j}(t)-p_{s-j}(t-p)$.  Thus, the claim will follow if we can show that
\begin{equation}
\label{eqn:neededbound}
\op{err}(t) = l(t) + o(t),
\end{equation}
for $t \in S$, for some linear polynomial $l$.  This follows by combining \eqref{eqn:rewrittensum} and Lemma~\ref{lem:poonenlemma} (note that we are fixing the equivalence class of $t$, mod $C$).

\end{proof}

Given this claim, we can complete the proof.  Since $\frac{t_i-\tilde{s}}{u} \in (a_0,a_0+\epsilon)$, by Claims 1 and 2, and \eqref{eqn:rationalfunctionequation}, the rational function
$a(t_i)/p_{s-\tilde{s}}(t_i)$ must have a horizontal asymptote as $t_i \to \infty$.   Now choose some collection of $\hat{t}_i-\tilde{s} \in S$ such that $\lbrace
(\hat{t}_i-\tilde{s})/u\rbrace \in (a_0-\epsilon,a_0)$  and $p_{s - \tilde{s}}(\hat{t}_i) \ne 0$.
By again appealing to Claims $1$ and $2$ and \eqref{eqn:rationalfunctionequation}, the rational function
$a(\hat{t}_i)/p_{s -\tilde{s}}(\hat{t}_i)$ has a horizontal asymptote as $\hat{t}_i$ goes to $+\infty$.  Since $a_0 >
\lbrace \frac{\hat{t}_i-\tilde{s}}{u} \rbrace$ while $a_0  < \lbrace \frac{t_i-\tilde{s}}{u}
\rbrace$, these two asymptotes are different, by Step 5; this can not happen for a rational function.  This is a contradiction.

{\em Step 6.}   We can therefore assume that $\tilde{\beta}$ is an
integer.  We can also assume that $q \ge 2$, or else
$\tilde{\alpha}\tilde{\beta}$ would be an integer.  We know that
$\lbrace \tilde{\beta} (t-\tilde{s}) \rbrace = 0$, and $\lbrace
\tilde{\beta}(t-\tilde{s})-\tilde{\alpha}) \rbrace$ is some nonzero number $b_0$
independent of $t$. Essentially the same argument in the previous step
can now be used to conclude that there is some small $\epsilon > 0$,
such that if $\lbrace \frac{t-\tilde{s}-\tilde{\alpha}}{u} \rbrace \in (b_0 -
\epsilon,b_0+\epsilon),$ then $\op{head}(t-\tilde{s})$ will differ based on
whether or not $b_0 > \lbrace \frac{t-\tilde{s}-\tilde{\alpha}}{u} \rbrace$, and
$\op{head}(t-\tilde{s}\pm1),\ldots,\op{head}(t-\tilde{s}\pm s)$ do not depend on $t$.  The fact
that $\lbrace \frac{t-\tilde{s}-\tilde{\alpha}}{u} \rbrace$ is dense as $t-\tilde{s}$
ranges over $S$ now gives a contradiction by essentially repeating the
arguments from the end of the previous step.  This proves
Claim~\ref{clm:importantclaim}.

\end{proof}

{\em Step 7.}  The arguments in the previous steps have proved the
first bullet point, since the sequence of values of a quasipolynomial
is $\mathcal{P}$-recursive.   This, together with
Claim~\ref{clm:importantclaim} now implies
Theorem~\ref{thm:precursion}.

For the second bullet point, if $I_{\mathcal{T}_{u,v}}(t)$ is
primitive and pseudo-integral, then $(u,v)$ must be conjugate
admissible quadratic irrationalities, and in particular
$I_{\mathcal{T}_{u,v}}(t)$ must be given by \eqref{eqn:explicitform}.
Then setting $t=0$ (or any multiple of $\alpha$) gives
\begin{equation}
\label{eqn:constantterm}
\frac{z\beta(\alpha-z)+2\alpha\sigma(t)}{2\alpha} = 1
\end{equation}
for all equivalence classes $z$, by \eqref{eqn:explicitform}.  Now
choose $t$ such that $t \equiv 1$ (mod $\alpha$).  By
\eqref{eqn:constantterm}, this gives
\[ \frac{\beta(\alpha-1)}{2\alpha}=1,\]
so $\beta = 2\alpha/(\alpha-1)$.  The only solutions to this equation
with $\alpha > 1$ and $\beta$ an integer are $(3,3)$ and $(2,4)$.

Conversely, if $\alpha=1$, or $(\alpha,\beta) \in \lbrace (3,3), (2,4) \rbrace$, then $\frac{z\beta(\alpha-z)+2\alpha\sigma(t)}{2\alpha} = 1$ for all equivalence classes $z$, hence the result follows, again by \eqref{eqn:explicitform}.

This completes the proof of Theorem~\ref{thm:irrationalperiodcollapse}.

\end{proof}

We can now give the proof that was owed for Example~\ref{ex:smallesttriangle}.

\begin{proof}

Let $(u,v)$ be admissible.  Then $uv=\frac{\alpha}{\beta}$.  To
minimize the area of $\mathcal{T}_{u,v}$, we would like to maximize
$\frac{\alpha}{\beta}$.  By
Theorem~\ref{thm:irrationalperiodcollapse}, if $ \alpha> 1$, then
$(\alpha,\beta) \in \lbrace (3,3), (2,4) \rbrace$.  The largest
possible value in these two cases is $\frac{\alpha}{\beta}=1$, which
is uniquely obtained by the ``golden mean" triangle.

For $\alpha=1$, the only possible value of $\beta$ that could give an
equally large $\frac{\alpha}{\beta}$ is when $\beta=1$.  There are no
real numbers satisfying $u+v=1, 1/u+1/v=1$, however.
\end{proof}

\begin{remark}
Similar arguments can be used to give another characterization of the
``golden mean" triangle: it is the only pseudo-integral triangle in
the family $\mathcal{T}_{u,v}$ where $u$ and $v$ are quadratic
irrational algebraic integers.  We omit the proof for brevity.

\end{remark}

\subsection{Properties of admissible irrational triangles}

We now prove some properties of the Ehrhart functions of admissible
triangles that mirror properties from the rational case; compare
\cite[\S 3]{Beck-Robins}, \cite{bs}.  For simplicity, we state some of
the results for pseudo-integral triangles, although we expect they
should hold more generally.  Along those lines, recall from
\cite[Lem. 3.9]{Beck-Robins} that if $\mathcal{T}_{u,v}$ is
pseudo-integral, then we can write
\[ \sum_{t \ge 0} I_{\mathcal{T}_{u,v}}(t) z^t = \frac{g_{u,v}(z)}{(1-z)^3},\]
where $g_{u,v}(z)$ is a polynomial of degree at most $2$.

\begin{proposition}  Let $\mathcal{T}_{u,v}$ be admissible.  Then the Ehrhart function of $\mathcal{T}_{u,v}$ satisfies
\label{prop:properties}
\begin{itemize}
\item  (Reciprocity)  If $t$ is positive, then
\[ I_{\mathcal{T}_{u,v}}(-t) = \lbrace \#(t\mathcal{T}^o_{u,v} \cap \mathbb{Z}^{2}) \rbrace+\mu(t),\]
where the superscript $^o$ denotes the interior, and the function
$\mu(t)$ is defined by $\mu(t)=0$ if $\alpha|t$, and $1$ otherwise.
\item  (Nonnegativity) If $\mathcal{T}_{u,v}$ is pseudo-integral, then
  each coefficient of $g_{u,v}$ is nonnegative.
\item  (Monotonicity)  If $\mathcal{T}_{u,v}$ and
  $\mathcal{T}_{u',v'}$ are both pseudo-integral, and
  $\mathcal{T}_{u,v} \subset \mathcal{T}_{u',v'}$, then each
  coefficient of $g_{u,v}$ is less than or equal to the corresponding
  coefficient of $g_{u',v'}$.
\end{itemize}

\end{proposition}

\begin{proof}  For the first bullet point, note that the number of
  lattice points on the boundary of the triangle with vertices
  $(t/u,0), (0,t/v)$ and $(0,0)$ is $\lfloor t/u \rfloor + \lfloor t/v
  \rfloor,$ plus the number of points on the slant edge.  We know that
\[ \lfloor t/u \rfloor + \lfloor t/v \rfloor =  t\beta - 1,\]
and we know that $ux+vy = t$ only if $x=y$.  The first bullet point
now follows by subtracting the number of lattice points on the
boundary from the formula in \eqref{eqn:explicitform}.

For the second and third bullet points, note that if
$g_{u,v}(z)=a_0+a_1z+a_2z^2,$ then $a_0 = 1,$ $a_1 =
I_{\mathcal{T}_{u,v}}(1)-3$, and $a_2 = 3 -
3I_{\mathcal{T}_{u,v}}(1)+I_{\mathcal{T}_{u,v}}(2)$ (here, we are
implicitly using the fact that $I_{\mathcal{T}_{u,v}}(0)=1$, as can be
seen by \eqref{eqn:constantterm}.)  To show the second bullet point,
we therefore first have to show that $I_{\mathcal{T}_{u,v}}(1) \geq 3,$
or equivalently, by \eqref{eqn:explicitform}, that
$\beta/\alpha+\beta \ge 4.$  If $(\alpha,\beta) \in \lbrace (3,3),
(2,4) \rbrace$, then this holds, so by
Theorem~\ref{thm:irrationalperiodcollapse} we can assume that
$\alpha=1$.  Since there are no quadratic irrationalities with
$\alpha=1$ and $\beta \le 3,$ nonnegativity for $a_1$ follows.  For
nonnegativity of $a_2$, we need to show that
$(\beta-\alpha\beta)/\alpha \ge -2.$  If $\alpha=1$, then this is
automatic; if $(\alpha,\beta) \in \lbrace (3,3), (2,4) \rbrace$, then
it holds as well.  This proves the second bullet point.

For the third bullet point, we are given that $1/u \le 1/u'$ and $1/v \le 1/v'$.  That $a_0 \le a'_0$ and $a_1 \le a'_1$ are immediate; to see that $a_2 \le a'_2$, we need to show that $(1/\alpha)(\beta-1) \le (1/\alpha')(\beta'-1)$.  This follows from $1/u \le 1/u', 1/v \le 1/v'.$
\end{proof}

\subsection{Examples in other dimensions}

Here we briefly mention some examples of polytopes in other dimensions that are not rational, but nevertheless have Ehrhart functions that are polynomials.

In dimension $1$, such examples are easy to come by:

\begin{example}
\label{ex:1dex}

Let $\mathcal{P} = [u,v] \subset \mathbb{R},$ where $u$ and $v$ are irrational numbers with $u - v = m$ and $m$ is an integer.  Then for positive integer $t$, $I_{\mathcal{P}}(t)=tm$

\end{example}

\begin{proof} We know that $I_{\mathcal{P}}(t) = \lfloor t u \rfloor - \lfloor tv \rfloor = tm$.

\end{proof}

In higher dimensions, we do not currently know many examples of truly
different character.  However, one has:

\begin{example}
\label{ex:ndex}
Let $\mathcal{P}$ denote the polytope with vertices
\[(0,0,\ldots,0), (1,0,\ldots,0), (0,1,\ldots,0), \]
\[(0,0,1,\ldots,0), \ldots , (0,\ldots,u,0), (0,\ldots,0,v),\]
where $(u,v)$ are positive admissible quadratic irrationalities.  Then $I_{\mathcal{P}}(t)$ is a polynomial in $t$.
\end{example}

A more interesting example is given in dimension three:

\begin{example}
\label{ex:3dex}
The polytope with vertices $(0,0,0)$, $(\frac 12,0,0)$,
$(0,2+\sqrt{2},0)$ and $(0,0,2-\sqrt{2})$ has an
Ehrhart function which is a polynomial.
\end{example}

The proof of Example~\ref{ex:ndex} is immediate; we defer the proof of
Example~\ref{ex:3dex} to \S\ref{sec:tetra}.

\section{Rational examples}

We now give the proof of Theorem~\ref{thm:main theorem-1}.

\begin{proof}

Firstly, it is easy to see that $\gcd(rq,ps)=1$. Let $\xi_m$ denote a
primitive $m$th root of unity.  By \cite[Theorem 2.10]{Beck-Robins},
we have
\begin{align}
\nonumber I_{\mathcal{T}_{q/p, s/r}}&=\frac{1}{2\cdot rq\cdot
  ps}(t\cdot pr)^2+\frac{1}{2}(t\cdot
pr)\left(\frac{1}{rq}+\frac{1}{ps}+\frac{1}{rq\cdot ps}\right)\\[5pt]
\nonumber
&+\frac{1}{4}\left(1+\frac{1}{rq}+\frac{1}{ps}\right)
+\frac{1}{12}\left(\frac{rq}{ps}+\frac{ps}{rq}+\frac{1}{rq\cdot
  ps}\right)\\[5pt]
\nonumber
&+\frac{1}{rq}\sum_{j=1}^{rq-1}\frac{\xi_{rq}^{-jt\cdot
    pr}}{(1-\xi_{rq}^{j\cdot ps})
(1-\xi_{rq}^{j})}+\frac{1}{ps}\sum_{l=1}^{ps-1}\frac{\xi_{ps}^{-lt\cdot pr}}{(1-\xi_{ps}^{l\cdot rq})(1-\xi_{ps}^{l})}\\[5pt]
\nonumber
&=\frac{pr}{2qs}\cdot
t^2+\frac{1}{2}\left(\frac{p}{q}+\frac{r}{s}+\frac{1}{qs}\right)
t+\frac{1}{4}\left(1+\frac{1}{rq}+\frac{1}{ps}\right)\\
\nonumber
& \quad\quad + \frac{1}{12}\left(\frac{rq}{ps}+\frac{ps}{rq}
+\frac{1}{rq\cdot ps}\right)\\[5pt]
\label{equation-1}
&+\frac{1}{rq}\sum_{j=1}^{rq-1}\frac{\xi_{q}^{-jtp}}{(1-\xi_{rq}^{j\cdot ps})(1-\xi_{rq}^{j})}+\frac{1}{ps}\sum_{l=1}^{ps-1}\frac{\xi_{s}^{-ltr}}{(1-\xi_{ps}^{l\cdot rq})(1-\xi_{ps}^{l})}.
\end{align}
Then it suffices to show that
\[
\frac{1}{ps}\sum_{l=1}^{ps-1}\frac{\xi_{s}^{-ltr}}{(1-\xi_{ps}^{l\cdot
    rq})(1-\xi_{ps}^{l})}
\]
is a constant function in $t$. In fact, writing $l=is+u: 0\leq i <p,
0\leq u<s$ and using the fact that $rq\equiv -1$ (mod p), we have
\begin{align}
\nonumber\frac{1}{ps}&\sum_{l=1}^{ps-1}\frac{\xi_{s}^{-ltr}}{(1-\xi_{ps}^{l\cdot
    rq})(1-\xi_{ps}^{l})}\\[5pt]
\nonumber
&=\frac{1}{ps}\sum_{i=1}^{p-1}\frac{1}{(1-\xi_{p}^{irq})(1-\xi_{p}^{i})}
+\frac{1}{ps}\sum_{u=1}^{s-1}\xi_{s}^{-utr}\sum_{i=0}^{p-1}
\frac{1}{(1-\xi_{ps}^{(is+u)rq})(1-\xi_{ps}^{is+u})}\\[5pt]
\label{equation-2}
&=\frac{1}{ps}\sum_{i=1}^{p-1}\frac{1}{(1-\xi_{p}^{irq})(1-\xi_{p}^{i})}
+\frac{1}{ps}\sum_{u=1}^{s-1}\xi_{s}^{-utr}\sum_{i=0}^{p-1}
\frac{1}{(1-\xi_{ps}^{urq-is})(1-\xi_{ps}^{u+is})}
\end{align}
Keeping in mind that $s|(rq+1)$ and
$\gcd\left(\frac{rq+1}{p},s\right)=1$, we find that
\[
ps\nmid u(rq+1),\ \ {\rm and}\ \ ps|u(rq+1)p
\]
for any $1\leq u\leq s-1$. By \cite[Lemma 2.1]{Gardiner-Kleinman}, we deduce that
\begin{align*}
\sum_{i=0}^{p-1}
\frac{1}{(1-\xi_{ps}^{urq-is})(1-\xi_{ps}^{u+is})}=0
\end{align*}
for any $1\leq u\leq s-1$. Therefore, we have
\begin{align*}
\frac{1}{ps}&\sum_{l=1}^{ps-1}\frac{\xi_{s}^{-ltr}}{(1-\xi_{ps}^{l\cdot rq})(1-\xi_{ps}^{l})}=\frac{1}{ps}\sum_{i=1}^{p-1}\frac{1}{(1-\xi_{p}^{irq})(1-\xi_{p}^{i})},
\end{align*}
which is a constant function in $t$. It follows from \eqref{equation-1} that $q$ is a quasi-period of $I_{\mathcal{T}_{q/p, s/r}}(t)$.

\end{proof}

Next we prove Corollary~\ref{pseudo-Corollary}.

\begin{proof}

By Theorem~\ref{thm:main theorem-1}, condition \eqref{pseudo-condition-1} implies that $q$ is a quasi-period of $I_{\mathcal{T}_{q/p, s/r}}(t)$. On the other hand, it is obvious that
\[
(x,y)\mapsto (y,x)
\]
is a bijection between lattice points in triangles $\mathcal{T}_{q/p,s/r}$ and $\mathcal{T}_{s/r,q/p}$. So
we have
\[
I_{\mathcal{T}_{q/p, s/r}}(t)=I_{\mathcal{T}_{s/r, q/p}}(t).
\]
By Theorem~\ref{thm:main theorem-1} again, condition
\eqref{pseudo-condition-2} means that $s$ is also a quasi-period of
$I_{\mathcal{T}_{q/p, s/r}}(t)$. Clearly, $\gcd(q,s)=1$. It follows
that $1$ is a quasi-period of $I_{\mathcal{T}_{q/p,
    s/r}}(t)$. Therefore, the triangle $I_{\mathcal{T}_{q/p, s/r}}(t)$
is a pseudo-integral triangle.

\end{proof}

\subsection{The case where $u=1/v$}

As mentioned in the introduction, for when $s=p$ and $r=q$, one can
also give a sufficient condition for a version of
Theorem~\ref{thm:main theorem-1}.  Specifically, we have:

\begin{theorem}\label{thm2}
Suppose that $p, q$ are relatively prime positive integers.  Then $q$
is a quasiperiod of $I_{\mathcal{T}_{q/p,p/q}}(t)$ if and only if
\begin{align}\label{collapse-condition-2}
p | (q^2+1)\qquad {\rm and}\qquad \gcd\left(\frac{q^2+1}{p},p\right)=1.
\end{align}
\end{theorem}
\begin{proof} Clearly, the ``if " part follows from the proof of
  Theorem~\ref{thm:main theorem-1}. We now proceed to the proof of the
  ``only if" part.

Suppose that $q$ is a quasi-period of $I_{\mathcal{T}_{q/p,p/q}}(t)$. By
\eqref{equation-1}, we deduce that
\[f_{p,q}(t):=\frac{1}{p^2}\sum_{l=1}^{p^2-1}\frac{\xi_{p}^{-ltq}}{(1-\xi_{p^2}^{lq^2})(1-\xi_{p^2}^{l})}\]
is a periodic function of $t$ with period $q$. Clearly, $p$ is also a
period of $f_{p,q}(t)$. Since $(p,q)=1$, we deduce that $f_{p,q}(t)$
is a constant function of $t$.
It follows from \eqref{equation-2} that
\begin{align*}
\frac{1}{p^2}\sum_{u=1}^{p-1}\xi_{p}^{-utq}\sum_{i=0}^{p-1}
\frac{1}{(1-\xi_{p^2}^{uq^2-ip})(1-\xi_{p^2}^{u+ip})}=C
\end{align*}
for some constant $C$. Keeping in mind the fact that $\gcd(p,q)=1$, we have
\[
\{\xi_{p}^{-utq}: 1\leq u\leq p-1\}=\{\xi_{p}^{jt}: 1\leq j\leq p-1 \}.
\]
So
$$\{1,\xi_{p}^{-tq},\xi_{p}^{-2tq},\ldots,\xi_{p}^{-(p-1)tq}\}
=\{1,\xi_{p}^{t},\xi_{p}^{2t},\ldots,\xi_{p}^{(p-1)t}\},$$
 which consists of $p$ linearly independent functions from
 $\mathbb{N}$ to $\mathbb{C}$. Hence we have
\[
\sum_{i=0}^{p-1}
\frac{1}{(1-\xi_{p^2}^{uq^2-ip})(1-\xi_{p^2}^{u+ip})}=0
\]
for any $1\leq u\leq p-1$. By applying \cite[Lemma 2.1]{Gardiner-Kleinman}, we deduce that
\[
p | u(q^2+1)\qquad {\rm and}\qquad p^2\nmid u(q^2+1)
\] for each $1\leq u\leq p-1$. Now we can conclude immediately that
\[
p |(q^2+1)\qquad {\rm and}\qquad \gcd\left(p,\frac{q^2+1}{p}\right)=1.
\]
\end{proof}

\subsection{The k-Fibonacci numbers}

In \cite{Gardiner-Kleinman}, it was shown that $\mathcal{T}_{q/p,p/q}$
is pseudo-integral if and only if $p=q=1$ or
$\{p,q\}=\{F_{2k-1},F_{2k+1}\}$ for
some $k\geq 1$, where $p$ and $q$ are relatively prime positive
integers, $F_m$ denotes the $m$th Fibonacci number.  We
now further study the relationship between the period collapse problem
and recursive sequences, by proving a similar result, involving two
consecutive terms in the sequence of generalized Fibonacci numbers.

Recall first for any integer $k\geq 1$, the $k$-Fibonacci sequence $\{F_n(k)\}$, defined recursively as follows:
\[
F_0(k)=0,F_1(k)=1,\ F_n(k)=kF_{n-1}(k)+F_{n-2}(k)\ \ (n\geq 2).
\]
Clearly, when $k=1$, we get the Fibonacci sequence. For notational simplicity, for any $k\geq 1, n\geq 2$, we let
$$
I_{k,n}(t):=I_{\mathcal{T}_{F_{n}(k)/F_{n-1}(k),F_{n-1}(k)/F_n(k)}}(t).
$$
In the following we shall consider quasi-period collapse in $I_{k,n}(t)$. To this end, we need the following immediate facts:

{\bf Fact 1}: For any $k,n\geq 1$, $\gcd(F_n(k), F_{n-1}(k))=1$ and $\gcd(F_n(k),k)=1$.

{\bf Fact 2}: For any $k,n\geq 1$, we have
\[
F_n(k)^2-kF_{n-1}(k)F_{n}(k)-F_{n-1}(k)^2+(-1)^n=0.
\]
Both Fact 1 and Fact 2 can be verified immediately by induction on $n$. We only give the proof of Fact 2 here. Clearly, the fact holds for $n=1$. Assume $n\geq 2$ and Fact 2 is true for $n-1$ and then we have
\begin{align*}
F_{n}&(k)^2-kF_{n-1}(k)F_{n}(k)-F_{n-1}(k)^2+(-1)^n\\[5pt]
&=(kF_{n-1}(k)+F_{n-2}(k))^2-kF_{n-1}(k)(kF_{n-1}(k)+F_{n-2}(k))-F_{n-1}(k)^2+(-1)^n\\[5pt]
&=kF_{n-1}(k)F_{n-2}(k)+F_{n-2}(k)^2-F_{n-1}(k)^2+(-1)^n\\[5pt]
&=-(F_{n-1}(k)^2-kF_{n-1}(k)F_{n-2}(k)-F_{n-2}(k)^2+(-1)^{n-1})\\
&=0.
\end{align*}

It follows from Fact 1 and Fact 2 that, when $n$ is even, both $(p,q)=(F_{n-1}(k), F_n(k))$ and $(p,q)=(F_{n+1}(k),F_n(k))$ satisfy condition \eqref{collapse-condition-2}.  We therefore get:

\begin{theorem}\label{period-collapse-3}
For any $k\geq 1$ and even integer $n\geq 2$, $F_n(k)$ is a common quasi-period of $I_{k,n}(t)$ and $I_{k,n+1}(t)$.
\end{theorem}

\subsection{Tetrahedra}
\label{sec:tetra}

We now give a few higher dimensional examples of period collapse.

Recall first the sequence given by $a_1=2, a_2=3, a_3=10, a_4=17$ and
\begin{equation}
\label{eqn:recursion2}
a_n=6a_{n-2}-a_{n-4}.
\end{equation}
It follows from \cite[Thm. 1.6.i]{Gardiner-Kleinman} that for each
$n\geq 1$, the triangle with vertices $(0,0)$,
$\left(\frac{a_{2n+1}}{a_{2n}},0\right)$ and
$\left(0,\frac{2a_{2n}}{a_{2n+1}}\right)$ is a pseudo-integral
triangle with Ehrhart polynomial
\[
I_n(t)=(t+1)^2.
\]

Using this we can show:
\begin{theorem}\label{OEIS-3}
Let $\{a_n\}$ be the sequence defined by \eqref{eqn:recursion2}.  Then
for any $n\geq 1$, the tetrahedron $T_n$ with vertices $(0,0,0)$,
$\left(\frac{1}{2},0,0\right)$,
$\left(0,\frac{a_{2n+1}}{a_{2n}},0\right)$ and
$\left(0,0,\frac{2a_{2n}}{a_{2n+1}}\right)$ is a pseudo-integral
tetrahedron.
\end{theorem}
\begin{proof} Let $f_n(t)$ denote the Ehrhart function for $T_n$. Then
  for any positive integer $t$, we have
\begin{align*}
f_n(t)&=\#\{(x,y,z)\in\mathbb{Z}^{3}\st
2x+\frac{a_{2n}}{a_{2n+1}}y+\frac{a_{2n+1}}{2a_{2n}}z\leq
t,\ x,y,z\geq 0\}\\[5pt]
&=\sum_{x=0}^{\lfloor\frac{t}{2}\rfloor}\#\{(y,z)\in\mathbb{Z}^2\st
\frac{a_{2n}}{a_{2n+1}}y+\frac{a_{2n+1}}{2a_{2n}}z\leq t-2x,\ y,z\geq
0\}\\[5pt]
&=\sum_{x=0}^{\lfloor \frac{t}{2}\rfloor}I_{n}(t-2x)=\sum_{x=0}^{\lfloor\frac{t}{2}\rfloor}(t-2x+1)^2\\[5pt]
&=\frac{1}{6}t^3+t^2+\frac{11}{6}t+1,
\end{align*}
where $I_n(t)$ denotes the Ehrhart function of the triangle with
vertices $(0,0),\left(\frac{a_{2n+1}}{a_{2n}},0\right)$ and
$\left(0,\frac{2a_{2n}}{a_{2n+1}}\right)$.
\end{proof}

We now give the proof of Example~\ref{ex:3dex}.  Note that
Example~\ref{ex:3dex} is natural to consider, in view of
Theorem~\ref{OEIS-3}.  It is easy to show that
\[
\lim_{n\rightarrow\infty} \frac{a_{2n+1}}{a_{2n}}=2+\sqrt{2},\ \ {\rm and}\ \ \lim_{n\rightarrow\infty} \frac{2a_{2n}}{a_{2n+1}}=2-\sqrt{2}.
\]
Theorem~\ref{OEIS-3} states that for any $n\geq 1$, the tetrahedron with vertices
$$(0,0,0), \left(\frac{1}{2},0,0\right), \left(0,\frac{a_{2n+1}}{a_{2n}},0\right)\ \ {\rm and},\  \left(0,0,\frac{2a_{2n}}{a_{2n+1}}\right)$$
is a pseudo-integral tetrahedron with the  Ehrhart polynomial
\[
f(t)=\frac{1}{6}t^3+t^2+\frac{11}{6}t+1,
\]
which is independent of $n$.  Thus, it is reasonable to expect that
the Ehrhart function of the limiting tetrahedron $T$ with vertices
$(0,0,0), \left(\frac{1}{2},0,0\right), \left(0,2+\sqrt{2},0\right)$
and $\left(0,0,2-\sqrt{2}\right)$ is also equal to the polynomial
$f(t)$, and one can indeed show this by using Theorem~\ref{OEIS-3}
plus a limiting argument.  We instead give a more direct proof that
does not require Theorem~\ref{OEIS-3}:

\begin{proposition}
The Ehrhart function of the irrational tetrahedron $T$ with vertices
$(0,0,0), \left(\frac{1}{2},0,0\right), \left(0,2+\sqrt{2},0\right)$
and $\left(0,0,2-\sqrt{2}\right)$ is the polynomial
$f(t)=\frac{1}{6}t^3+t^2+\frac{11}{6}t+1$.
\end{proposition}
\begin{proof} Let $g(t)$ denote the Ehrhart function of $T$. Then we have
\begin{align*}
&g(t)=\#(tT\cap \mathbb{Z}^3)\\[5pt]
=&\#\{(x,y,z)\in\mathbb{Z}^3\st 2x+\left(1-\frac{\sqrt{2}}{2}\right)y+\left(1+\frac{\sqrt{2}}{2}\right)z\leq t, x,y,z\geq 0\}\\[5pt]
=&\sum_{i=0}^{\lfloor\frac{t}{2}\rfloor}\#\{(i,y,z)\st (i,y,z)\in tT\cap\mathbb{Z}^3\}\\[5pt]
=&\sum_{i=0}^{\lfloor\frac{t}{2}\rfloor}\#\{(i,y,z)\st (i,y,z)\in
tT\cap\mathbb{Z}^3, y\geq z\}\\
& \qquad +\sum_{i=0}^{\lfloor\frac{t}{2}\rfloor}\#\{(i,y,z)\st
(i,y,z)\in tT\cap\mathbb{Z}^3,y<z\}\\[5pt]
=&\ \sum_{i=0}^{\lfloor\frac{t}{2}\rfloor}\sum_{j=0}^{\lfloor\frac{t-2i}{2}\rfloor}\#\{(i,y,j)\st (i,y,j)\in tT\cap\mathbb{Z}^3, y\geq j\}\\[5pt]
&\qquad +\sum_{i=0}^{\lfloor\frac{t}{2}\rfloor}\sum_{k=0}^{\lfloor\frac{t-2i-1+\frac{\sqrt{2}}{2}}{2}\rfloor}\#\{(i,k,z)\st
(i,k,z)\in tT\cap\mathbb{Z}^3, k<z\}\\[5pt]
=&\sum_{i=0}^{\lfloor\frac{t}{2}\rfloor}\sum_{j=0}^{\lfloor\frac{t-2i}{2}\rfloor}
\left(\left\lfloor\frac{t-2i-2j}{1-\frac{\sqrt{2}}{2}}\right\rfloor+1\right)
+\sum_{i=0}^{\left\lfloor\frac{t}{2}\right\rfloor}\sum_{k=0}^{\left\lfloor\frac{t-2i-1+\frac{\sqrt{2}}{2}}{2}\right\rfloor}
\left\lfloor\frac{t-2k-2i}{1+\frac{\sqrt{2}}{2}}\right\rfloor\\[5pt]
=&\sum_{i=0}^{\lfloor\frac{t}{2}\rfloor}\sum_{j=0}^{\lfloor\frac{t-2i}{2}\rfloor}
\left(\left\lfloor\frac{t-2i-2j}{1-\frac{\sqrt{2}}{2}}\right\rfloor+1+\left\lfloor\frac{t-2i-2j}{1+\frac{\sqrt{2}}{2}}\right\rfloor\right)
\end{align*}
Keeping in mind that
\[
\frac{1}{1-\frac{\sqrt{2}}{2}}=2+\sqrt{2},\ \ \frac{1}{1-\frac{\sqrt{2}}{2}}=2-\sqrt{2},
\]
we deduce that
\begin{align*}
\left\lfloor\frac{t-2i-2j}{1-\frac{\sqrt{2}}{2}}\right\rfloor&=\lfloor(t-2i-2j)(2+\sqrt{2})\rfloor\\[5pt]
&=\lfloor(t-2i-2j)(4-(2-\sqrt{2}))\rfloor\\[5pt]
&=4t-8i-8j+\left\lfloor-\frac{t-2i-2j}{1+\frac{\sqrt{2}}{2}}\right\rfloor.
\end{align*}
So we have
\begin{align*}
g(t)&=\left\{
       \begin{array}{ll} \displaystyle
         \frac{t}{2}+1+\sum_{i=0}^{\lfloor\frac{t}{2}\rfloor}\sum_{j=0}^{\lfloor\frac{t-2i}{2}\rfloor}
(4t-8i-8j), & \hbox{$t$\ {\rm is\ even};} \\[15pt]
         \displaystyle\sum_{i=0}^{\lfloor\frac{t}{2}\rfloor}\sum_{j=0}^{\lfloor\frac{t-2i}{2}\rfloor}
(4t-8i-8j), & \hbox{$t$\ {\rm is\ odd}}
       \end{array}
     \right.\\[7pt]
&=\frac{1}{6}t^3+t^2+\frac{11}{6}t+1.
\end{align*}
\end{proof}

\end{document}